\documentclass[leqno,12pt]{amsart}
\usepackage{amsfonts}
\usepackage{amsmath,amssymb,amsthm}

\everymath{\displaystyle}

\newcommand{\E}{\mathbb E}
\newcommand{\R}{\mathbb R}
\newcommand{\tr}{\mathrm{tr}}
\newcommand{\ds}{\displaystyle}
\newcommand{\manifold}[1]{\mathcal{#1}}
\newcommand{\M}{\manifold{M}}
\newcommand{\D}{\manifold{D}}

\newtheorem{theorem}{Theorem}[section]

\theoremstyle{definition}

\theoremstyle{remark}
\newtheorem{rem}[theorem]{Remark}

\numberwithin{equation}{section}

\begin{document}

\title[Timelike Surfaces in the Minkowski 4-Space]{Fundamental Theorems for Timelike Surfaces in the Minkowski 4-Space}

\author{Victoria Bencheva, Velichka Milousheva}

\address{Institute of Mathematics and Informatics, Bulgarian Academy of Sciences,
Acad. G. Bonchev Str. bl. 8, 1113, Sofia, Bulgaria}
\email{viktoriq.bencheva@gmail.com}
\email{vmil@math.bas.bg}


\begin{abstract}

In the present paper, we study timelike surfaces free of minimal points in the four-dimensional Minkowski space. For each such surface we introduce a geometrically determined 
pseudo-orthonormal frame field and writing the derivative formulas with respect to this moving frame field and using the integrability conditions, we obtain a system of six functions satisfying some natural conditions. In the general case, we prove a Fundamental Bonnet-type theorem (existence and uniqueness theorem) stating that these six functions, satisfying the natural conditions, determine the surface up to a motion. In some particular cases, we reduce the number of functions and give the fundamental theorems.

\end{abstract}

\maketitle 

\section{Introduction}

In the local theory of surfaces in Euclidean and pseudo-Euclidean spaces, one of the basic problems is to determine the surface by a system of some functions satisfying some differential equations. This is the fundamental Bonnet-type theorem giving the natural conditions under which the surface is determined up to a motion. In the Euclidean space $\R^4$, the general fundamental theorem states that each surface free of minimal points is determined up to a motion in  $\R^4$ by eight invariant functions satisfying some natural conditions (differential equations) \cite{GM-1}. 
For the class of minimal surfaces in $\R^4$, the number of the invariant functions  and the number of the differential equations determining the surfaces are reduced to two \cite{Trib-Guad}. 
The surfaces with parallel normalized mean curvature vector field are determined uniquely up to a motion  by three invariant functions satisfying a system of three partial differential equations \cite{G-M-Fil}.

Similar results hold for spacelike surfaces in the Minkowski 4-space $\R^4_1$. The local theory of spacelike surfaces in $\R^4_1$ whose mean
curvature vector at any point is a non-zero spacelike vector or timelike vector, is developed in \cite{GM-2}. This class of surfaces is  determined up to a motion in $\R^4_1$ by eight invariant functions 
satisfying some natural conditions. Spacelike surfaces in $\R^4_1$ whose
mean curvature vector at any point is a lightlike vector are called marginally
trapped surfaces. These surfaces were defined by Roger Penrose in order to
study the global properties of spacetime \cite{Pen}  and play an
important role in the theory of cosmic black holes. Recently,  marginally trapped surfaces satisfying some extra conditions have
been studied intensively from a mathematical viewpoint \cite{Chen-Veken-1}, \cite{Chen-Veken-2}, \cite{Chen-Veken-3}, etc. The invariant theory of marginally trapped surfaces is developed in \cite{GM-3}, where it is proved that the marginally trapped surfaces in $\R^4_1$ are determined up to a motion by seven invariant functions. 
Maximal spacelike and minimal timelike surfaces in $\R^4_1$ are studied in  \cite{Al-Pal} and \cite{G-M-IJM}, respectively, and it is proved that their geometry is determined by two invariant functions, satisfying a system of two partial differential equations.

\vskip 1mm
In the present paper, we study timelike surfaces free of minimal points in the Minkowski 4-space $\R^4_1$. For each such surface we introduce a
pseudo-orthonormal frame field $\{x,y,n_1,n_2\}$, which is geometrically determined by the two lightlike directions  in the tangent space of the surface and the  mean curvature vector field $H$. We call this  pseudo-orthonormal frame field  a \textit{geometric frame field} of the surface. Writing the derivative formulas with respect to the geometric frame field and using the integrability conditions, we obtain a system of six functions satisfying some differential equations (natural conditions). We prove a Fundamental Bonnet-type theorem stating that, in the general case, these six functions, satisfying the natural conditions, determine the surface up to a motion in $\R^4_1$ (Theorem \ref{Th:FundTh6f}). Then we give the fundamental theorems in some particular cases (Theorem \ref{Th:FundTh6f3} and Theorem \ref{Th:FundTh6f2}).

\section{Preliminaries}

Let $\mathbb R^4_1$  be the four-dimensional Minkowski space  endowed with the standard flat metric
$\langle ., . \rangle$ of signature $(3,1)$ given in local coordinates by
$dx_1^2 + dx_2^2 + dx_3^2 -dx_4^2$. The considerations in this paper are local and all functions are supposed to be of class $C^{\infty}$.

Let $\M^2 = (\D , z)$ be a surface in $\R^4_1$, where $\D\subset\R^2$ and $z : \D \to \R^4_1$ is an immersion, i.e. 
$\M^2$ is locally parametrized by  $\M^2: z = z(u,v), \, \, (u,v) \in \D$. 

A surface $\M^2$ in  $\mathbb R^4_1$ is said to be:

\hskip 6mm 
- \textit{timelike}, if the restriction of $\langle ., . \rangle$ to each tangent space of $\M^2$ is indefinite;
 
\hskip 6mm 
- \textit{spacelike}, if the restriction of $\langle ., . \rangle$  to each tangent space of $\M^2$ is positive definite;

\hskip 6mm 
- \textit{lightlike}, if the restriction of $\langle ., . \rangle$  to each tangent space of $\M^2$ is degenerate.

This paper is devoted to the study of timelike surfaces in $\mathbb R^4_1$, so $\langle ., . \rangle$ induces a Lorentzian metric $g$ on $\M^2$. 

We denote by $\widetilde{\nabla}$ and $\nabla$ the Levi Civita connections on $\mathbb R^4_1$ and $\M^2$, respectively. So, we have the following formulas of Gauss and Weingarten 
$$\begin{array}{l}
\vspace{2mm}
\widetilde{\nabla}_xy = \nabla_xy + \sigma(x,y);\\
\vspace{2mm}
\widetilde{\nabla}_x \xi = - A_{\xi} x + D_x \xi,
\end{array}$$
where  $x$ and $y$ are vector fields tangent to $\M^2$ and $\xi$ is a normal vector field. These formulas determine 
the second fundamental tensor $\sigma$, the normal connection $D$
and the shape operator $A_{\xi}$ with respect to $\xi$. In general, $A_{\xi}$ is not diagonalizable.

The mean curvature vector  field $H$ of $\M^2$ is defined as
$$H = \displaystyle{\frac{1}{2}\,  \tr\, \sigma}.$$ 
If the mean curvature vector vanishes identically, i.e.  $H=0$, the surface is called \textit{minimal}.

It is well known that for a timelike surface $\M^2$ in $\R^4_1$, locally there exists a coordinate system  $(u,v)$ such that  the metric tensor $g$ of $\M^2$ has the following form \cite{Lar}:
\begin{equation*} \label{E:Eq-g}
g= - f^2(u, v)(du\otimes dv + dv\otimes du)
\end{equation*}
for some positive function $f(u, v)$. We suppose that $z=z(u, v), (u, v) \in \mathcal{D}$ is such a local parametrization on $\M^2$. 
Then, the coefficients of the first fundamental form are
\begin{equation*}
E = \langle z_u, z_u \rangle = 0; \quad F = \langle z_u, z_v \rangle = - f^2(u, v); \quad G = \langle z_v, z_v \rangle = 0,
\end{equation*}
where $z_u$ and $z_v$ denote the derivatives of the vector function $z(u,v)$, i.e. 
$z_u=\frac{\partial z}{\partial u},  \; 	z_v=\frac{\partial z}{\partial v}$.
Since $\langle z_u, z_u \rangle = 0$ and $\langle z_v, z_v \rangle = 0$, the parameters $(u,v)$ are called isotropic parameters of the surface.

Let us consider the pseudo-orthonormal tangent frame field  of $\M^2$ defined by $x=\displaystyle{\frac{z_u}{f}}$, $y=\displaystyle{\frac{z_v}{f}}$. Obviously,
 $\langle x, x \rangle = 0$,  $\langle x, y \rangle = -1$, $\langle y, y \rangle = 0$. Hence, the mean curvature vector field $H$ is given by 
$$H = - \sigma(x, y).$$

Since we consider surfaces free of minimal points, i.e.  $H \neq 0$ at all points, we can choose a unit normal vector field $n_1$ which is collinear with $H$, i.e. $H = \nu n_1$  for a smooth function $\nu = || H ||$. Then, $\sigma (x,y) = - \nu n_1$.
We consider the unit normal vector field $n_2$ such that $\{n_1, n_2\}$ is an orthonormal frame field of the normal bundle ($n_2$ is determined up to orientation). So, we can write the following formulas for the second fundamental tensor $\sigma$:
\begin{equation}
\begin{array}{l} \label{E:Eq-1}
\vspace{2mm}
\sigma (x,x) = \lambda_1 n_1 + \mu_1 n_2; \\
\vspace{2mm}
\sigma (x,y) = -\nu n_1; \\
\vspace{2mm}
\sigma (y,y) = \lambda_2 n_1 + \mu_2 n_2,
\end{array}
\end{equation} 
where $\lambda_1, \mu_1, \lambda_2, \mu_2$ are  smooth functions defined by:
$$
\lambda_1 = \langle \widetilde{\nabla}_x x, n_1 \rangle; \quad  \mu_1 = \langle \widetilde{\nabla}_x x, n_2 \rangle; \quad
\lambda_2 = \langle \widetilde{\nabla}_y y, n_1 \rangle; \quad  \mu_2 = \langle \widetilde{\nabla}_y y, n_2 \rangle.  
$$
Having in mind that $x=\ds{\frac{z_u}{f}}$, $y=\ds{\frac{z_v}{f}}$ and using $\langle z_u, z_u \rangle = 0, \langle z_u, z_v \rangle = -f^2(u,v),  \langle z_v, z_v \rangle = 0$,
after differentiation we obtain:
\begin{equation} 
\begin{array}{l} \label{E:Eq-2}
\vspace{2mm}
\nabla _x x = \displaystyle{\frac{f_u}{f^2} \,x};  \\
\vspace{2mm}
\nabla _x y = \quad\qquad -\displaystyle{\frac{f_u}{f^2} \,y}; \\
\vspace{2mm}
\nabla _y x = \displaystyle{-\frac{f_v}{f^2} \,x}; \\
\vspace{2mm}
\nabla _y y = \quad\qquad\,\,\, \displaystyle{\frac{f_v}{f^2} \,y}. 
\end{array}
\end{equation}
Denoting $\gamma_1 = \frac{f_u}{f^2} = x(\ln f)$ and $\gamma_2 = \frac{f_v}{f^2} = y(\ln f)$, from \eqref{E:Eq-1} and \eqref{E:Eq-2} we obtain the following derivative formulas:
\begin{equation}\label{E:DerivFormIsotr}
\begin{array}{l}
\vspace{2mm}
\widetilde{\nabla} _x x = \gamma_1 x \qquad\quad + \lambda_1 n_1 + \mu_1 n_2; \\
\vspace{2mm}
\widetilde{\nabla}_x y = \quad\quad -\gamma_1 y -\nu n_1; \\
\vspace{2mm}
\widetilde{\nabla}_y x = -\gamma_2 x \quad\quad -\nu n_1; \\
\vspace{2mm}
\widetilde{\nabla}_y y = \quad\quad\,\,\, \gamma_2 y \, \, +  \lambda_2 n_1 + \mu_2 n_2.
\end{array}
\end{equation}

For the normal frame field $\{n_1, n_2\}$ by use of \eqref{E:DerivFormIsotr} we can  derive the formulas:
\begin{equation}\label{E:DerivNormFormIsotr}
\begin{array}{l}
\vspace{2mm}
\widetilde{\nabla}_x n_1 = -\nu x  + \lambda_1 y \quad\quad + \beta_1 n_2; \\
\vspace{2mm}
\widetilde{\nabla}_y n_1 = \lambda_2 x -\nu y \quad\quad\,\,\,\,\, + \beta_2 n_2; \\
\vspace{2mm}
\widetilde{\nabla}_x n_2 = \quad\quad + \mu_1 y  -\beta_1 n_1; \\
\vspace{2mm}
\widetilde{\nabla}_y n_2 = \mu_2 x \quad\quad\,\,\,\, -  \beta_2 n_1,
\end{array}
\end{equation}
where $\beta_1 = \langle \widetilde{\nabla}_x n_1 , n_2 \rangle$ and  $\beta_2 = \langle \widetilde{\nabla}_y n_1 , n_2 \rangle$.

\begin{rem}
The pseudo-orthonormal frame field $\{x,y,n_1,n_2\}$ is geometrically determined: $x,y$ are the two lightlike directions in the tangent space (determined up to notation); $n_1$ is the unit normal vector field collinear with the mean curvature vector field $H$; $n_2$ is determined by the condition that  $\{n_1, n_2\}$ is an orthonormal frame field of the normal bundle ($n_2$ is determined up to a sign). We call this  pseudo-orthonormal frame field $\{x,y,n_1,n_2\}$ a \textit{geometric frame field} of the surface \cite{BM-2}.
\end{rem}

Formulas
\eqref{E:DerivFormIsotr} and \eqref{E:DerivNormFormIsotr} are the derivative formulas of the surface with respect to the geometric frame field $\{x,y,n_1,n_2\}$. 

The functions $\beta_1$ and  $\beta_2$ characterize the class of surfaces with parallel mean curvature vector field. In \cite{BM-2}, we proved that: 

\begin{itemize}
\item
The timelike surface $\M^2$ has parallel mean curvature vector field if and only if $\beta_1 = \beta _2 =0$ and $\nu =const$.
\item
The timelike surface $\M^2$ has parallel normalized mean curvature vector field if and only if $\beta_1 = \beta _2 =0$ and $\nu \neq const$.
\end{itemize}

In the present paper we study timelike surfaces in $\R^4_1$ with $\beta_1^2 + \beta _2^2 \neq 0$.

\section{Fundamental theorems}

Let $\M^2: z=z(u, v), (u, v) \in \mathcal{D}, \, \D\subset\R^2$ be a local parametrization with respect to isotropic parameters on a timelike surface  $\M^2$ in $\R^4_1$ and $\{x,y,n_1,n_2\}$ be the geometric pseudo-orthonormal frame  field introduced above.

Since the Levi Civita  connection $\widetilde{\nabla}$ of $\R^4_1$ is flat, we have 
\begin{equation} \label{E:Eq-3}
\widetilde{R}(x,y,x) = 0; \quad \widetilde{R}(x,y,y) = 0; \quad \widetilde{R}(x,y,n_1) = 0; \quad \widetilde{R}(x,y,n_2) = 0,
\end{equation}
where 
$$\widetilde{R}(x,y,z) = \widetilde{\nabla}_x \widetilde{\nabla}_y z  - \widetilde{\nabla}_y \widetilde{\nabla}_x z - \widetilde{\nabla}_{[x,y]} z$$ 
 for arbitrary vector fields $x, y, z$. Using \eqref{E:Eq-3} and the derivative formulas \eqref{E:DerivFormIsotr} and \eqref{E:DerivNormFormIsotr}, we 
 obtain the following integrability conditions:
\begin{equation}\label{E:integrCondIsotr1_2}
x(\lambda_2)+y(\nu) + 2 \gamma_1 \lambda_2 - \mu_2 \beta_1=0; 
\end{equation}
\begin{equation}\label{E:integrCondIsotr2_2}
x(\nu)+y(\lambda_1) + 2 \gamma_2 \lambda_1 - \mu_1 \beta_2=0; 
\end{equation}
\begin{equation}\label{E:integrCondIsotr3_2}
x(\mu_2) + 2\gamma_1 \mu_2 + \nu \beta_2 + \lambda_2 \beta_1=0; 
\end{equation}
\begin{equation}\label{E:integrCondIsotr4_2}
y(\mu_1) + 2\gamma_2 \mu_1 + \nu \beta_1 + \lambda_1 \beta_2=0; 
\end{equation}
\begin{equation}\label{E:integrCondIsotr5_2}
x(\gamma_2)+y(\gamma_1) + 2 \gamma_1 \gamma_2 - \nu^2 + \lambda_1 \lambda_2 + \mu_1 \mu_2=0; 
\end{equation}
\begin{equation}\label{E:integrCondIsotr6_2}
x(\beta_2)-y(\beta_1)  + \mu_1 \lambda_2-\lambda_1 \mu_2  + \gamma_1 \beta_2 - \gamma_2 \beta_1=0. 
\end{equation}

\begin{rem} 
If we assume that both $\mu_1$ and $\mu_2$ are zero functions, i.e. $\mu_1(u,v)= 0$ and  $\mu_2(u,v)= 0$ for all $(u, v) \in \mathcal{D}$, then the surface consists of inflection points, i.e. at each point $p \in \M^2$ the first normal space  ${\rm Im} \, \sigma_p = {\rm span} \{\sigma(x, y): x, y \in T_p \M^2 \}$ is one-dimensional.  
In such case, the surface is developable or lies in a 3-dimensional space \cite{Lane}.
\end{rem}

So, further we consider timelike surfaces in $\R^4_1$ that are free of inflection  points, i.e. we assume that $\mu_1^2 + \mu_2^2 \neq 0$ at least in a sub-domain $\mathcal{D}_0$ of $\mathcal{D}$.
Without loss of generality we may assume that $\mu_1 \neq 0$. 
We will consider separately the two cases: 
\begin{itemize}
\item
$\mu_1 \neq 0$ and $\mu_2 \neq 0$ in a sub-domain; 
\item
$\mu_1 \neq 0$ and $\mu_2 = 0$ in a sub-domain.
\end{itemize}

\vskip 2mm
\textbf{Case I}. $\mu_1 \neq 0$ and $\mu_2 \neq 0$ in a sub-domain.
\vskip 1mm

In this case we call the surface a timelike surfaces of \textit{first type}.

From \eqref{E:integrCondIsotr1_2} and \eqref{E:integrCondIsotr2_2} we express the functions $\beta_1$ and $\beta_2$ as follows:
\begin{equation} \label{E:Eq-4}
\beta_1 = \frac{1}{\mu_2} \Bigl ( x(\lambda_2) + y(\nu) + 2 \gamma_1 \lambda_2 \Bigr ); \qquad 
\beta_2 = \frac{1}{\mu_1} \Bigl ( x(\nu) + y(\lambda_1) + 2 \gamma_2 \lambda_1 \Bigr ). 
\end{equation}
Having in mind that $x = \frac{1}{f} \frac{\partial}{\partial u}, \, y = \frac{1}{f} \frac{\partial}{\partial v}$ and 
$\gamma_1 = \frac{f_u}{f^2}, \, \gamma_2 = \frac{f_v}{f^2}$, from \eqref{E:Eq-4} we obtain
\begin{equation*} \label{E:Eq-5}
\beta_1 = \frac{1}{f\mu_2} \Bigl ( (\lambda_2)_u + \nu_v + \lambda_2 (\ln f^2)_u \Bigr ); \qquad 
\beta_2 = \frac{1}{f\mu_1} \Bigl ( \nu_u + (\lambda_1)_v + \lambda_1 (\ln f^2)_v \Bigr ).
\end{equation*}
Hence, the functions  $\beta_1$ and $\beta_2$ are expressed by $f, \nu, \lambda_1, \lambda_2, \mu_1, \mu_2$. So, we have six functions $f, \nu, \lambda_1, \lambda_2, \mu_1, \mu_2$ satisfying the following four equations, which are derived from \eqref{E:integrCondIsotr3_2}, \eqref{E:integrCondIsotr4_2}, \eqref{E:integrCondIsotr5_2}, and \eqref{E:integrCondIsotr6_2}, respectively:
\begin{equation*}\label{E:integrCondIsotr3_case1}
(\lambda^2_2+\mu_2^2)_u + (\ln f^4)_u (\lambda^2_2+\mu_2^2) + 2 \lambda_2 \nu_v + \frac{2 \nu \mu_2}{\mu_1} \Bigl ( \nu_u + (\lambda_1)_v + \lambda_1 (\ln f^2)_v \Bigr )=0;
\end{equation*}
\begin{equation*}\label{E:integrCondIsotr4_case1}
(\lambda_1^2+\mu_1^2)_v + (\ln f^4)_v (\lambda^2_1+\mu_1^2) + 2 \lambda_1 \nu_u + \frac{2 \nu \mu_1}{\mu_2} \Bigl ((\lambda_2)_u +  \nu_v + \lambda_2 (\ln f^2)_u \Bigr )=0;
\end{equation*}
\begin{equation*}\label{E:integrCondIsotr5_case1}
\frac{2 f f_{uv} - 2 f_u f_v}{f^4} + \lambda_1 \lambda_2 + \mu_1 \mu_2 - \nu^2 =0;
\end{equation*}
\begin{equation*}\label{E:integrCondIsotr6_case1}
\begin{array}{l}
(\mu_1 \lambda_2 - \mu_2 \lambda_1) \left ( f^2 \mu_1 \mu_2 - (\ln f^2)_{uv} \right ) + \nu_{uu} \mu_2 - \nu_{vv} \mu_1 + (\lambda_1)_{uv} \mu_2 - (\lambda_2)_{uv} \mu_1 +\\
+ \, \mu_2 (\lambda_1)_u (\ln f^2)_v -  \mu_1 (\lambda_2)_v (\ln f^2)_u - \frac{\mu_2(\mu_1)_u}{\mu_1} \left ( \nu_u + (\lambda_1)_v + \lambda_1 (\ln f^2)_v \right ) +\\ 
+ \, \frac{\mu_1 (\mu_2)_v}{\mu_2} \left ( (\lambda_2)_u + \nu_v +  \lambda_2 (\ln f^2)_u \right ) = 0.
\end{array}
\end{equation*}

\vskip 1mm
Now we can give the fundamental theorem:

\begin{theorem}\label{Th:FundTh6f}
Let $f(u,v) > 0$, $\nu(u,v)$, $\lambda_1(u,v)$, $\mu_1(u,v)$, $\lambda_2(u,v)$, $\mu_2(u,v)$, $\mu_1 \mu_2 \neq 0$ be six smooth functions, 
 defined in a domain
${\mathcal D}, \,\, {\mathcal D} \subset {\R}^2$, and satisfying the conditions
\begin{equation} \label{E:FundTh6f} 
\begin{array}{l}
\vspace{2mm}
\text{(i)} \;\; (\lambda^2_2+\mu_2^2)_u + (\ln f^4)_u (\lambda^2_2+\mu_2^2) + 2 \lambda_2 \nu_v + \frac{2 \nu \mu_2}{\mu_1} \Bigl ( \nu_u + (\lambda_1)_v + \lambda_1 (\ln f^2)_v \Bigr )=0;
\\
\vspace{2mm}
\text{(ii)} \;\; (\lambda_1^2+\mu_1^2)_v + (\ln f^4)_v (\lambda^2_1+\mu_1^2) + 2 \lambda_1 \nu_u + \frac{2 \nu \mu_1}{\mu_2} \Bigl ((\lambda_2)_u +  \nu_v + \lambda_2 (\ln f^2)_u \Bigr )=0;
\\
\vspace{2mm}
\text{(iii)} \;\; \frac{2 f f_{uv} - 2 f_u f_v}{f^4} + \lambda_1 \lambda_2 + \mu_1 \mu_2 - \nu^2 =0;
\\
\text{(iv)} \;\;
\begin{array}{l}
(\mu_1 \lambda_2 - \mu_2 \lambda_1) \left ( f^2 \mu_1 \mu_2 - (\ln f^2)_{uv} \right ) + \nu_{uu} \mu_2 - \nu_{vv} \mu_1 + (\lambda_1)_{uv} \mu_2 - (\lambda_2)_{uv} \mu_1 +\\
+ \,  \mu_2 (\lambda_1)_u (\ln f^2)_v - \mu_1 (\lambda_2)_v (\ln f^2)_u - \frac{\mu_2(\mu_1)_u}{\mu_1} \left ( \nu_u + (\lambda_1)_v + \lambda_1 (\ln f^2)_v \right ) +\\ 
+ \, \frac{\mu_1 (\mu_2)_v}{\mu_2} \left ( (\lambda_2)_u + \nu_v +  \lambda_2 (\ln f^2)_u \right ) = 0.
\end{array}
\end{array} 
\end{equation}
If $\{x_0,  y_0,  (n_1)_0, (n_2)_0\}$ is a pseudo-orthonormal frame at
a point $p_0 \in \R^4_1$, then there exists a subdomain ${\mathcal D}_0 \subset {\mathcal D}$
and a unique timelike surface of first type
$\M^2: z = z(u,v), \,\, (u,v) \in {\mathcal D}_0$, parametrized by isotropic parameters, such that $\M^2$ passes through $p_0$, $\{x_0, y_0,  (n_1)_0, (n_2)_0\}$ is the geometric
frame of $\M^2$ at $p_0$. 
\end{theorem}

\begin{proof}
Using the given functions we define  $\gamma_1 =  \ds \frac{f_u}{f^2}$, $\gamma_2 =  \ds \frac{f_v}{f^2}$, $\beta_1 = \ds \frac{(\lambda_2)_u + \nu_v + \lambda_2 (\ln f^2)_u}{f\mu_2}$, $\beta_2 = \ds \frac{\nu_u + (\lambda_1)_v +  \lambda_1 (\ln f^2)_v}{f\mu_1}$ and consider the following system of partial differential equations for the unknown vector
functions $x = x(u,v), \, y = y(u,v), \, n_1 = n_1(u,v), \,n_2 = n_2(u,v)$
in $\R^4_1$:
\begin{equation}
\begin{array}{ll} \label{E:Th6fEqSystem1}
\vspace{2mm}
x_u = f \left(\gamma_1\, x + \lambda_1\, n_1  + \mu_1\, n_2\right)
& \quad x_v = f \left( -\gamma_2\, x - \nu\, n_1\right)\\
\vspace{2mm}
y_u = f \left(- \gamma_1\, y - \nu\, n_1 \right)
& \quad y_v = f \left( \gamma_2\, y + \lambda_2 \, n_1 + \mu_2 \, n_2 \right)\\
\vspace{2mm}
(n_1)_u = f \left( - \nu\, x + \lambda_1\, y + \beta_1 \, n_2\right) &
\quad (n_1)_v =  f \left( \lambda_2\, x - \nu\, y +\beta_2 \, n_2\right) \\
\vspace{2mm}
(n_2)_u = f \left(  \mu_1\, y - \beta_1 \, n_1 \right) &
\quad (n_2)_v =  f \left(  \mu_2\, x - \beta_2 \, n_1\right)
\end{array}
\end{equation}
For convenience, we denote
$$\mathcal{W} =
\left(%
\begin{array}{c}
\vspace{2mm}
  x \\
  \vspace{2mm}
  y \\
\vspace{2mm}
  n_1 \\
  n_2 \\
\end{array}%
\right)\!\!; \;\;
\mathcal{A} = f \left(%
\begin{array}{cccc}
\vspace{2mm}
  \gamma_1 & 0 & \lambda_1  & \mu_1 \\
\vspace{2mm}
  0 & -\gamma_1 & -\nu & 0 \\
\vspace{2mm}
  -\nu & \lambda_1 & 0 & \beta_1 \\
  0 & \mu_1 & -\beta_1 & 0 \\
\end{array}%
\right)\!\!; \;\;
\mathcal{B} = f   
\left(%
\begin{array}{cccc}
\vspace{2mm}
  -\gamma_2 & 0 & -\nu & 0 \\
\vspace{2mm}
  0 & \gamma_2 &  \lambda_2 &  \mu_2 \\
\vspace{2mm}
   \lambda_2 & -\nu & 0 & \beta_2 \\
   \mu_2 & 0 & -\beta_2 & 0 \\
\end{array}%
\right)\!.$$
Then, system \eqref{E:Th6fEqSystem1} can be written in the following matrix form:
\begin{equation}
\begin{array}{l} \label{E:Th6fEqMatrixSystem1}
\vspace{2mm}
\mathcal{W}_u = \mathcal{A}\,\mathcal{W},\\
\vspace{2mm} 
\mathcal{W}_v = \mathcal{B}\,\mathcal{W}.
\end{array}
\end{equation}
The integrability conditions of system \eqref{E:Th6fEqMatrixSystem1} are $\mathcal{W}_{uv} = \mathcal{W}_{vu}$,
i.e.
\begin{equation} \label{E:Th6fEqMatrixSystemElements1}
\displaystyle{\frac{\partial a_i^k}{\partial v} - \frac{\partial b_i^k}{\partial u}
+ \sum_{j=1}^{4}(a_i^j\,b_j^k - b_i^j\,a_j^k) = 0, \quad i,k = 1,
\dots, 4,}
\end{equation}
 where by $a_i^j$ and $b_i^j$ ($i, j = 1,\dots, 4$) we denote  the
elements of the matrices $\mathcal{A}$ and $\mathcal{B}$, respectively. Taking into consideration the conditions given in  \eqref{E:FundTh6f}, one can check that
 equalities \eqref{E:Th6fEqMatrixSystemElements1} are fulfilled. Hence, there exists a subdomain
$\mathcal{D}_1 \subset \mathcal{D}$ and unique vector functions $x
= x(u,v), \, y = y(u,v), \,n_1 = n_1(u,v)$, $n_2 = n_2(u,v), \,\, (u,v)
\in \mathcal{D}_1$, which satisfy system \eqref{E:Th6fEqSystem1} and the initial conditions
$$x(u_0,v_0) = x_0, \quad y(u_0,v_0) = y_0, \quad n_1(u_0,v_0) = (n_1)_0, \quad n_2(u_0,v_0) = (n_2)_0.$$
In order to prove that the vector functions $x(u,v), \, y(u,v), \,n_1(u,v), \,n_2(u,v)$ form
a pseudo-orthonormal frame in $\R^4_1$ for each $(u,v) \in
\mathcal{D}_1$, we consider the  functions:
$$\begin{array}{lllll}
\vspace{2mm}
  h_1 = \langle x,x \rangle; & \quad h_3 = \langle n_1, n_1 \rangle - 1;  & \quad h_5 =
  \langle x,y \rangle +1; & \quad  h_7 =   \langle x,n_2 \rangle;  & \quad h_9 = \langle y,n_2 \rangle; \\  
\vspace{2mm}
  h_2 = \langle y, y \rangle; & \quad h_4 = \langle n_2,n_2 \rangle - 1; & \quad  h_6 = \langle x,n_1 \rangle; & \quad  h_8 = \langle y,n_1 \rangle; 
 & \quad  h_{10} = \langle n_1,n_2 \rangle; 
\end{array}$$
defined for $(u,v) \in \mathcal{D}_1$. Having in mind  that $x(u,v), \,
y(u,v), \,n_1(u,v), \,n_2(u,v)$ satisfy \eqref{E:Th6fEqSystem1}, we obtain  the system
\begin{equation} \label{E:Th6fEqLinearSystem1}
\displaystyle{\frac{\partial h_i}{\partial u} = k_i^j \, h_j}, \qquad 
\displaystyle{\frac{\partial h_i}{\partial v} = m_i^j \, h_j}; \qquad i = 1, \dots, 10,
\end{equation}
where $k_i^j, m_i^j, \,\, i,j = 1, \dots, 10$ are
functions of $(u,v) \in \mathcal{D}_1$.  System \eqref{E:Th6fEqLinearSystem1} is a linear
system of partial differential equations for the functions
$h_i(u,v)$, satisfying the conditions $h_i(u_0,v_0) = 0$ for all $i = 1,
\dots, 10$, since $\{x_0, \, y_0, \, (n_1)_0,\, (n_2)_0\}$ is a pseudo-orthonormal frame. Therefore, $h_i(u,v) = 0, \,\,i = 1, \dots, 10$ for
each $(u,v) \in \mathcal{D}_1$. Hence, the vector functions
$x(u,v)$, $y(u,v)$, $n_1(u,v)$, $n_2(u,v)$ form a pseudo-orthonormal frame
in $\E^4_1$ for each $(u,v) \in \mathcal{D}_1$.

Now, we consider the following system of PDEs for the vector function
$z(u,v)$:
\begin{equation} \label{E:Th6fEqSystemForZ1}
\begin{array}{lll}
\vspace{2mm}
z_u = f \, x\\
 z_v = f \, y
\end{array}
\end{equation}
Using \eqref{E:FundTh6f} and \eqref{E:Th6fEqSystem1}, one can check that the integrability
conditions $z_{uv} = z_{vu}$ of system \eqref{E:Th6fEqSystemForZ1}
 are fulfilled. Consequently,  there exists a subdomain  $\mathcal{D}_0 \subset \mathcal{D}_1$ and
a unique vector function $z = z(u,v)$, $(u,v) \in \mathcal{D}_0$, satisfying $z(u_0, v_0) = p_0$.

Finally, we consider the surface $\M^2: z = z(u,v), \,\, (u,v) \in
\mathcal{D}_0$. Obviously, $\M^2$ is a timelike surface in $\R^4_1$ parametrized by isotropic parameters $(u,v)$, since $\langle z_u, z_u \rangle = 0$, $\langle z_v, z_v \rangle = 0$, $\langle z_u, z_v \rangle = - f^2(u,v)$. 

\end{proof}

\vskip 1mm
\textbf{Case II}. $\mu_1 \neq 0$ and $\mu_2 = 0$ in a sub-domain.
\vskip 1mm

In this case, we have the following subcases:
$$
\begin{array}{l} 
\vspace{2mm}
(a) \hspace{3mm} \lambda_2 \neq 0;  \\ 
(b) \hspace{3mm} \lambda_2 = 0.
\end{array}
$$

\vskip 1mm 
\textbf{Subcase II  (a)}:  $\lambda_2 \neq 0$.
\vskip 1mm

In this subcase we call the surface a timelike surface of \textit{second type}.

\vskip 1mm
Now, using the integrability conditions under the assumptions $\mu_2 = 0$  and $\lambda_2 \neq 0$, we obtain the following expressions for the functions $\beta_1$ and $\beta_2$:
$$
\beta_2 = \frac{1}{\mu_1 } (x(\nu)+y(\lambda_1) + 2 \gamma_2 \lambda_1) = \frac{1}{f \mu_1 } (\nu_u + (\lambda_1)_v + (\ln f^2)_v \lambda_1);
$$
$$
\beta_1 = -\frac{\nu \beta_2}{\lambda_2 } = - \frac{\nu}{f \mu_1 \lambda_2}(\nu_u + (\lambda_1)_v +  (\ln f^2)_v \lambda_1);
$$
and the next four equations:
$$
(\lambda_2)_u + \nu_v + (\ln f^2)_u \lambda_2 = 0;
$$
$$
\frac{2 f f_{uv} - 2 f_u f_v}{f^4} - (\nu^2 - \lambda_1 \lambda_2) = 0;
$$ 
$$
(\mu_1)_v + (\ln f^2)_v \mu_1 - \frac{\nu^2 -\lambda_1 \lambda_2}{\lambda_2 \mu_1} (\nu_u + (\lambda_1)_v + (\ln f^2)_v \lambda_1) = 0;
$$
$$
\begin{array}{l}
 \lambda_2 \Bigl (\! \nu_{uu} \!+ \!(\lambda_1)_{uv}\! +\! (\lambda_1)_u (\ln f^2)_v\! +\! \lambda_1  (\ln f^2)_{uv} \!\Bigr )  
\!+ \!\nu \Bigl (\! \nu_{uv} \!+ \!(\lambda_1)_{vv} \!+ \! (\lambda_1)_v (\ln f^2)_v \!+ \!\lambda_1  (\ln f^2)_{vv} \!\Bigr ) + \\ 
+ f^2 \mu_1^2 \lambda_2^2  - \Bigl ( \nu_u + (\lambda_1)_v + \lambda_1 (\ln f^2)_v \Bigr )\Bigl (\lambda_2 (\ln |\mu_1|)_u +\nu_v - \nu (\ln |\mu_1|)_v -\nu (\ln |\lambda_2|)_v \Bigr ) = 0.
\end{array}
$$

\vskip 1mm
Finally, we can formulate the fundamental theorem in this subcase:

\begin{theorem}\label{Th:FundTh6f3}
Let $f(u,v) > 0$, $\nu(u,v)$, $\lambda_1(u,v)$, $\mu_1(u,v)$, $\lambda_2(u,v) \neq 0$, be five smooth functions, 
 defined in a domain
${\mathcal D}, \,\, {\mathcal D} \subset {\R}^2$, and satisfying the conditions
\begin{equation*} \label{E:FundTh6f3} 
\begin{array}{l}
\vspace{2mm}
\text{(i)} \;\; (\lambda_2)_u + \nu_v + (\ln f^2)_u \lambda_2 = 0;
\\
\vspace{2mm}
\text{(ii)} \;\; (\mu_1)_v + (\ln f^2)_v \mu_1 - \frac{\nu^2 -\lambda_1 \lambda_2}{\lambda_2 \mu_1} (\nu_u + (\lambda_1)_v + (\ln f^2)_v \lambda_1) = 0;
\\
\vspace{2mm}
\text{(iii)} \;\; \frac{2 f f_{uv} - 2 f_u f_v}{f^4} - (\nu^2 - \lambda_1 \lambda_2) = 0; 
\\
\text{(iv)} \;\;
\begin{array}{l}
 \lambda_2 \Bigl (\! \nu_{uu} \!+ \!(\lambda_1)_{uv}\! +\! (\lambda_1)_u (\ln f^2)_v\! +\! \lambda_1  (\ln f^2)_{uv} \!\Bigr ) \! 
+ \!\nu \Bigl (\! \nu_{uv} \!+ \!(\lambda_1)_{vv} \!+ \! (\lambda_1)_v (\ln f^2)_v \!+ \!\lambda_1  (\ln f^2)_{vv} \!\Bigr ) + \\ 
+ f^2 \mu_1^2 \lambda_2^2  - \Bigl ( \nu_u + (\lambda_1)_v + \lambda_1 (\ln f^2)_v \Bigr )\Bigl (\lambda_2 (\ln |\mu_1|)_u +\nu_v - \nu (\ln |\mu_1|)_v -\nu (\ln |\lambda_2|)_v \Bigr ) = 0.
\end{array}
\end{array} 
\end{equation*}
If $\{x_0,  y_0,  (n_1)_0, (n_2)_0\}$ is a pseudo-orthonormal frame at
a point $p_0 \in \R^4_1$, then there exists a subdomain ${\mathcal D}_0 \subset {\mathcal D}$
and a unique timelike surface of second type
$\M^2: z = z(u,v), \,\, (u,v) \in {\mathcal D}_0$, parametrized by isotropic parameters, such that $\M^2$ passes through $p_0$ and $\{x_0,  y_0,  (n_1)_0, (n_2)_0\}$ is the geometric
frame of $\M^2$ at  $p_0$.
\end{theorem}

\begin{proof}
Let us denote $\beta_1 = - \frac{\nu}{f \mu_1 \lambda_2}(\nu_u + (\lambda_1)_v + (\ln f^2)_v \lambda_1)$, 
$\beta_2 =\frac{1}{f \mu_1 } (\nu_u + (\lambda_1)_v +  (\ln f^2)_v \lambda_1)$, $\gamma_1 =  \ds \frac{f_u}{f^2}$, $\gamma_2 =  \ds \frac{f_v}{f^2}$, 
 and consider the following system of partial differential equations for the unknown vector
functions $x = x(u,v), \, y = y(u,v), \, n_1 = n_1(u,v), \,n_2 = n_2(u,v)$
in $\R^4_1$:
\begin{equation}
\begin{array}{ll} \label{E:Th6f3EqSystem1}
\vspace{2mm}
x_u = f \left(\gamma_1\, x + \lambda_1\, n_1  +  \mu_1\, n_2\right)
& \quad x_v = f \left( -\gamma_2\, x - \nu\, n_1\right)\\
\vspace{2mm}
y_u = f \left(- \gamma_1\, y - \nu\, n_1 \right)
& \quad y_v = f \left( \gamma_2\, y + \lambda_2 \, n_1  \right)\\
\vspace{2mm}
(n_1)_u = f \left( - \nu\, x + \lambda_1\, y + \beta_1 \, n_2\right) &
\quad (n_1)_v =  f \left( \lambda_2\, x - \nu\, y +\beta_2 \, n_2\right) \\
\vspace{2mm}
(n_2)_u = f \left(   \mu_1 \, y - \beta_1 \, n_1 \right) &
\quad (n_2)_v =  f \left(  - \beta_2 \, n_1\right)
\end{array}
\end{equation}
We denote
$$\mathcal{W} =
\left(%
\begin{array}{c}
\vspace{2mm}
  x \\
  \vspace{2mm}
  y \\
  \vspace{2mm}
  n_1 \\
  \vspace{2mm}
  n_2 \\
\end{array}%
\right)\!\!; \;\;
\mathcal{A} = f \left(%
\begin{array}{cccc}
\vspace{2mm}
  \gamma_1 & 0 & \lambda_1  &  \mu_1 \\
  \vspace{2mm}
  0 & -\gamma_1 & -\nu & 0 \\
\vspace{2mm}
  -\nu & \lambda_1 & 0 & \beta_1 \\
\vspace{2mm}
  0 &  \mu_1 & -\beta_1 & 0 \\
\end{array}%
\right)\!\!; \;\;
\mathcal{B} = f   
\left(%
\begin{array}{cccc}
\vspace{2mm}
  -\gamma_2 & 0 & -\nu & 0 \\
\vspace{2mm}
  0 & \gamma_2 &  \lambda_2 &  0 \\
\vspace{2mm}
   \lambda_2 & -\nu & 0 & \beta_2 \\
\vspace{2mm}
   0 & 0 & -\beta_2 & 0 \\
\end{array}%
\right)\!$$
and rewrite system \eqref{E:Th6f3EqSystem1} in the matrix form:
\begin{equation*}
\begin{array}{l} \label{E:Th6f3EqMatrixSystem1}
\vspace{2mm}
\mathcal{W}_u = \mathcal{A}\,\mathcal{W},\\
\vspace{2mm}
\mathcal{W}_v = \mathcal{B}\,\mathcal{W}.
\end{array}
\end{equation*}
Further, the proof of the theorem follows the steps in the proof of Theorem \ref{Th:FundTh6f}.

\end{proof}

\vskip 1mm 
\textbf{Subcase II  (b)}:  $\lambda_2 = 0$.
\vskip 1mm

We call the surfaces satisfying $\mu_2 =0$  and  $\lambda_2 = 0$  timelike surfaces of \textit{third type}.
\vskip 1mm

In this subcase,  under the assumption $\lambda_2 = \mu_2 =0$ from the integrability conditions we obtain that $\nu = \nu(u)$, $\beta_2 =0$, 
$\beta_1 = -\frac{1}{\nu f}\left((\mu_1)_v + \mu_1 (\ln f^2)_v\right)$ and the following three equations:
\begin{equation}\label{E:integrCondIsotr2_1_2_1}
\nu_u + (\lambda_1)_v + \lambda_1 (\ln f^2)_v  = 0;
\end{equation}
$$
(\mu_1)_{vv} + (\mu_1)_v (\ln f^2)_{v}+ \mu_1 (\ln f^2)_{vv} = 0;
$$
$$
\frac{2f f_{uv} - 2f_u f_v}{f^4} = \nu^2.
$$
The last equality implies that the function $\nu$ is expressed by the function $f$ and its derivatives as follows:
\begin{equation}\label{E:integrCondIsotr2_1_2_1-a}
\nu^2 = f^{-2} \left(\ln f^2\right)_{uv}.
\end{equation}
Hence, the function $f^{-2} \left(\ln f^2\right)_{uv}$ depends only on the parameter $u$, since $\nu = \nu(u)$.
Using \eqref{E:integrCondIsotr2_1_2_1-a} we can write equality \eqref{E:integrCondIsotr2_1_2_1} in the form:
\begin{equation}\label{E:integrCondIsotr2_1_2_1-l}
(\lambda_1)_v + \lambda_1 (\ln f^2)_v   + \left(f^{-1} \sqrt{\left(\ln f^2\right)_{uv}}\right)_u= 0.
\end{equation}

Now, we can give the fundamental theorem in this subcase.

\begin{theorem}\label{Th:FundTh6f2}
Let $f(u,v) >0$, $\lambda_1(u,v)$ and $\mu_1(u,v)$ be three smooth functions, 
 defined in a domain
${\mathcal D}, \,\, {\mathcal D} \subset {\R}^2$, and satisfying the conditions
\begin{equation*} \label{E:FundTh6f2} 
\begin{array}{l}
\vspace{2mm}
\text{(i)} \;\; \left(f^{-2} \left(\ln f^2\right)_{uv}\right)_v =0;
\\
\vspace{2mm}
\text{(ii)} \;\; (\lambda_1)_v + \lambda_1 (\ln f^2)_v   + \left(f^{-1} \sqrt{\left(\ln f^2\right)_{uv}}\right)_u= 0;
\\
\vspace{2mm}
\text{(iii)} \;\; (\mu_1)_{vv} + (\mu_1)_v (\ln f^2)_v + \mu_1 (\ln f^2)_{vv} = 0.
\end{array} 
\end{equation*}
If $\{x_0, y_0,  (n_1)_0, (n_2)_0\}$ is a pseudo-orthonormal frame at
a point $p_0 \in \R^4_1$, then there exists a subdomain ${\mathcal D}_0 \subset {\mathcal D}$
and a unique timelike surface of third type
$\M^2: z = z(u,v), \, (u,v) \in {\mathcal D}_0$,  parametrized by isotropic parameters,  such that $\M^2$ passes through $p_0$, $\{x_0,  y_0,  (n_1)_0, (n_2)_0\}$ is the geometric
frame of $\M^2$ at the point $p_0$. 
\end{theorem}

\begin{proof}
Let us denote $\nu = f^{-1} \sqrt{\left(\ln f^2\right)_{uv}}$,  $\beta_1 = -\left((\ln f^2)_{uv}\right)^{-\frac{1}{2}}\left((\mu_1)_v + \mu_1 (\ln f^2)_v\right)$, 
$\gamma_1 =  \ds \frac{f_u}{f^2}$, $\gamma_2 =  \ds \frac{f_v}{f^2}$, 
 and consider the following system of partial differential equations for the unknown vector
functions $x = x(u,v), \, y = y(u,v), \, n_1 = n_1(u,v), \,n_2 = n_2(u,v)$
in $\R^4_1$:
\begin{equation*}
\begin{array}{ll} \label{E:Th6f2EqSystem1}
\vspace{2mm}
x_u = f \left(\gamma_1\, x + \lambda_1\, n_1  + \mu_1\, n_2\right)
& \quad x_v = f \left(-\gamma_2\, x - \nu\, n_1\right)\\
\vspace{2mm}
y_u = f \left(- \gamma_1\, y - \nu\, n_1 \right)
& \quad y_v = f \gamma_2\, y \\
\vspace{2mm}
(n_1)_u = f \left( - \nu\, x + \lambda_1\, y + \beta_1 \, n_2\right) &
\quad (n_1)_v =  - f \nu\, y  \\
\vspace{2mm}
(n_2)_u = f \left(  \mu_1\, y - \beta_1 \, n_1 \right) &
\quad (n_2)_v =  0
\end{array}
\end{equation*}

Further, the proof of the theorem follows the steps in the proof of Theorem \ref{Th:FundTh6f}.

\end{proof}

\vskip 1mm 
\textbf{Acknowledgments:}
The  authors are partially supported by the National Science Fund, Ministry of Education and Science of Bulgaria under contract KP-06-N52/3.

\end{document}